\documentclass[11pt]{article}
\usepackage{latexsym}
\usepackage{amsfonts}
\usepackage{enumerate,graphicx}
\usepackage{verbatim}

\usepackage{amsmath, amssymb, amsthm}

\theoremstyle{definition}

\numberwithin{equation}{section}

\newcommand\vanish[1]{}	% \vanish{TEXT} hides text in the compiled file.

\newcommand\ourcomment[1]{ \textbf{[#1]} }
\newcommand\oc\ourcomment

\newtheorem{theorem}{Theorem}[section]
\newtheorem{conjecture}{Conjecture}[section]
\newtheorem{lemma}{Lemma}[section]

\title{Perfect divisibility and $2$-divisibility}
\author{Vaidy Sivaraman\\
Binghamton University, Binghamton, NY 13902, USA
}

\begin{document}
\title{Four NP-complete problems about generalizations of perfect graphs}
\maketitle

%\author{Vaidy Sivaraman}

%\address{Department of Mathematical Sciences, Binghamton University.}
%\email{vaidy@math.binghamton.edu}

\begin{abstract}
We show that the following problems are NP-complete.  \\

\begin{itemize}
\item Can the vertex set of a graph be partitioned into two sets such that each set induces a perfect graph? \\
\item Is the difference between the chromatic number and clique number at most $1$ for every induced subgraph of a graph? \\
\item Can the vertex set of every induced subgraph of a graph be partitioned into two sets such that the first set induces a perfect graph, and the clique number of the graph induced by the second set is smaller than that of the original induced subgraph? \\
\item Does a graph contain a stable set whose deletion results in a perfect graph? \\
%$A,B$ such that G[A]$ is perfect and $\omega(G[B]) < \omega(G)$? 
\end{itemize}

The proofs of the NP-completeness of the four problems follow the same pattern: Showing that all the four problems are NP-complete when restricted to triangle-free graphs by using results of Maffray and Preissmann \cite{MP} on $3$-colorability and $4$-colorability of triangle-free graphs. 
\end{abstract}

\section{Introduction}
All graphs considered in this article are finite and simple.  Let $G$ be a 
graph. The complement $G^c$ of $G$ is the graph with vertex set $V(G)$ and 
such that two vertices are adjacent in $G^c$ if and only if they are 
non-adjacent in $G$. 
For two graphs $H$ and $G$, $H$ is an {\em induced subgraph} of $G$ if 
$V(H) \subseteq V(G)$, and a pair of vertices $u,v \in V(H)$ is adjacent if and only if it is adjacent in $G$. We say that $G$ {\em contains} $H$ if $G$ has an induced subgraph isomorphic to $H$. If $G$ does not contain $H$, we say that $G$ is {\em $H$-free}. For a set $X \subseteq V(G)$ we denote by  $G[X]$ the induced subgraph of $G$ with vertex set $X$. A {\em hole} in a graph is an induced subgraph that is isomorphic to  the cycle $C_k$ with $k\geq 4$, and $k$ is the {\em length} of the hole. A hole is {\em odd} if $k$ is odd, and {\em even} otherwise.  The chromatic number of a graph $G$ is denoted by $\chi(G)$ and the clique number by $\omega(G)$.  $G$ is called {\em perfect} if for every  induced subgraph $H$ of $G$, $\chi(H) = \omega (H)$.  $G$ is said to be {\em perfectly divisible} if for all induced subgraphs $H$ of $G$, $V(H)$ can be partitioned into two sets $A,B$ such that $H[A]$ is perfect and $\omega(B) < \omega(H)$.  $G$ is said to be nice if for every induced subgraph $H$ of $G$, $\chi(H) - \omega(H) \in \{0,1\}$.  $G$ is said to be {\em 2-perfect} if $V(G)$ can be partitioned into two sets $A, B$ such that both $G[A]$ and $G[B]$ are perfect.  $G$ is said to be stable-perfect if $G$ contains a stable set $S$ such that $G \setminus S$ is perfect. Note that perfect graphs are stable-perfect, and stable-perfect graphs are 2-perfect, perfectly divisible, and nice. In this note, we show that the recognition problems for the four classes ($2$-perfect, nice, perfectly divisible, stable-perfect) are NP-complete, a stark contrast to the existence of a polynomial-time recognition algorithm for perfect graphs \cite{CCLSV}.

\section{Four NP-complete problems}

We need the following results from \cite{MP}.

\begin{theorem}[Maffray-Preissmann]\label{THM1}
It is NP-complete to determine whether a triangle-free graph is $3$-colorable.
\end{theorem}

\begin{theorem}[Maffray-Preissmann]\label{THM2}
It is NP-complete to determine whether a triangle-free graph is $4$-colorable.
\end{theorem}

The following is a basic fact about perfect graphs.

\begin{lemma}\label{BASICLEMMA}
A triangle-free graph is perfect if and only if it is bipartite. 
\end{lemma}

\begin{proof}
Since bipartite graphs are perfect, one direction is trivial. To prove the other direction, let $G$ be a triangle-free perfect graph.  Since $G$ contains neither a triangle nor an odd hole, it contains no odd cycle as a subgraph. Hence $G$ is bipartite.  
\end{proof}

We first prove the NP-completeness of recognizing $2$-perfect graphs. First we need a lemma. 

\begin{lemma}\label{PerfectPartitionLemma}
A triangle-free graph is 2-perfect if and only if it is $4$-colorable.
\end{lemma}

\begin{proof}
This follows easily from Lemma  \ref{BASICLEMMA}.
\end{proof}

\begin{theorem}
It is NP-complete to determine whether a graph is $2$-perfect. 
\end{theorem}

\begin{proof}
We show that the restricted problem of determining whether a triangle-free graph is $2$-perfect is NP-complete. 
Let $G$ be a triangle-free graph. By Lemma \ref{PerfectPartitionLemma}, $G$ is $2$-perfect if and only if $G$ is $4$-colorable. 
By Theorem \ref{THM2} it is NP-complete to determine whether a triangle-free graph is $4$-colorable, We thus conclude that it is NP-complete to determine whether
a triangle-free graph is $2$-perfect.
\end{proof}

We now move on to the classes of perfectly divisible graphs, stable-perfect, and nice graphs. Problem 32 in \cite{VS} asks whether nice graphs can be recognized in polynomial time. The recognition problem for nice graphs turns out to be NP-complete. The following lemma tells that for triangle-free graphs, the three classes mentioned above are equivalent to the class of $3$-colorable graphs. 

\begin{lemma}\label{TFAE}
For a triangle-free graph $G$, the following are equivalent:
\begin{enumerate}[(i)]
\item $G$ is $3$-colorable.
\item $G$ is perfectly divisible.
\item $G$ is stable-perfect.
\item $G$ is nice.
\end{enumerate}
\end{lemma}

\begin{proof}
We prove the following chain of implications $(i) \Rightarrow (ii) \Rightarrow (iii) \Rightarrow (iv) \Rightarrow (i)$. \\

$(i) \Rightarrow (ii)$:  Suppose $G$ is $3$-colorable.  Let $H$ be an induced subgraph of $G$. Note that $H$ is also $3$-colorable. We may assume that $H$ has clique number $2$. Let $(S_1, S_2, S_3)$ be a partition of $V(H)$ into three stable sets. Now $(S_1 \cup S_2, S_3)$ is a partition of $V(G)$ as in the definition of being perfectly divisible. We conclude that $G$ is perfectly divisible. \\

$(ii) \Rightarrow (iii)$:  Suppose $G$ is perfectly divisible. Hence there is a partition of $V(G)$ into sets $A,B$ such that $G[A]$ is perfect and  $\omega(B) < \omega(G)$. Since $G$ has no triangles, this implies that $B$ is a stable set. Thus $G$ is stable-perfect. \\

$(iii) \Rightarrow (iv)$: Suppose $G$ is stable-perfect. Let $H$ be an induced subgraph of $G$. We may assume that $H$ has clique number $2$. Thus $H$ contains a stable set $S$ such that $H \setminus S$ is perfect. Since $H$ is also triangle-free, by Lemma \ref{BASICLEMMA}, $H \setminus S$ is bipartite.  Hence the chromatic number of $H$ is at most $3$. We conclude that $G$ is nice. \\

$(iv) \Rightarrow (i)$: Suppose $G$ is nice. Since $G$ is triangle-free, its clique number is at most $2$. Since $G$ is nice, we conclude that its chromatic number is at most $3$. Thus $G$ is $3$-colorable. \\

This concludes the proof of all the implications, and proves the theorem. 
\end{proof}

\begin{theorem}
The following problems are NP-complete: 
\begin{enumerate}
\item Given a graph, is it perfectly divisible?
\item Given a graph, is it stable-perfect?
\item Given a graph, is it nice?
\end{enumerate}
\end{theorem}

\begin{proof}
By Lemma \ref{TFAE} and Theorem \ref{THM1}, the problems are already NP-complete when restricted to triangle-free graphs. 
\end{proof}

\section{Open problems}

 $G$ is said to be {\em $2$-divisible} if for all induced subgraphs $H$ of $G$, $V(H)$ can be partitioned into two sets $A,B$ such that $\omega(A) < \omega(H)$ and $\omega(B) < \omega(H)$. 

\begin{conjecture}
It is NP-complete to determine whether a graph is $2$-divisible. 
\end{conjecture}

There is a nice conjecture about $2$-divisible graphs:

\begin{conjecture}[Hoang-McDiarmid \cite{HM}]
A graph is $2$-divisible if and only if it is odd-hole-free. 
\end{conjecture}

The complexity of the recognition of odd-hole-free graphs is also unknown. 

\begin{conjecture}
It is NP-complete to determine whether a graph contains an odd hole. 
\end{conjecture}

\section{Acknowledgment}
I would like to thank Maria Chudnovsky for some useful discussion about perfectly divisible graphs which inspired this note.

\end{document}